\documentclass[a4paper]{article}
\usepackage[round]{natbib}
\usepackage{amsmath, amsfonts}
\usepackage{amsthm}
\usepackage[title]{appendix}

\newtheorem{theorem}{Theorem}
\newtheorem{lemma}{Lemma}
\newtheorem{condition}{Condition}
\newtheorem{proposition}{Proposition}

\DeclareMathOperator*{\argmin}{arg\,min}

\title{On the error in Laplace approximations
of high-dimensional integrals}
\author{Helen Ogden}
\date{University of Southampton, UK}

\begin{document}

\maketitle
\subsubsection*{Summary}
Laplace approximations are
commonly used to approximate high-dimensional integrals
in statistical applications, but the quality of such
approximations as the dimension of the integral grows
is not well understood.
In this paper, we prove a new result on the size of the
error in first- and higher-order
Laplace approximations, and apply this result to investigate
the quality of Laplace approximations to the likelihood in some
generalized linear mixed models.
\\ \\\textbf{Keywords}: Asymptotic approximation; 
Intractable likelihood;
Generalized linear mixed model

\section{Introduction}
Integrals of the form
\begin{equation}
  \label{eq:form_of_integral}
 L =  \int_{\mathbb{R}^d} \exp\{-g({u})\} d{u}
\end{equation}
are frequently encountered in statistical applications, where $g(.)$ is a smooth function with a
unique minimum. For example, the likelihood 
function for a generalized linear mixed model is of this
form, where ${u}$ is a vector of random effects.
Integrals of this type are also common in Bayesian applications,
for example as marginal likelihoods 
used for model comparison.

Laplace approximations are often used to approximate
integrals of form  \eqref{eq:form_of_integral}.
 Suppose that $g({u})$ grows at rate $n$.
Often $g({u})$ is a sum with one term for each observation, so $n$
is the sample size. If $d$ is fixed
as $n \rightarrow \infty$, many results
on the quality of the Laplace approximation
are available: see \cite{Small2010} for a review.
However, in many examples of interest,
$d$ and $n$ tend to infinity simultaneously, and
there are very few results available on the quality
of Laplace approximations in this setting.

\cite{Shun1995} provide a formal expansion
for integrals of type \eqref{eq:form_of_integral}.
By studying the size of various terms in this
expansion, 
they conjecture that the first-order 
Laplace approximation should be reliable
if $d = o(n^{1/3})$, 
under the assumption that all derivatives of $g(.)$
grow at rate $n$. This condition is typically not met
for generalized linear mixed models,
and we give an example in which
$d = o(n^{1/3})$ but the error in the first-order Laplace
approximation grows with $n$.

In Section \ref{sec:main_result}, assuming alternative conditions on $g(.)$,
we develop a new result on the error
in Laplace approximations of various orders to integrals
of type \eqref{eq:form_of_integral}.
Our result is motivated by a two-level random
intercept model with $n_j$ observations on items in the $j$th cluster,
for which the likelihood
factorizes into a product of terms
\[L = \prod_{j=1}^d \int_{-\infty}^\infty \exp\{-g_j(u_j)\} du_j, \]
where each $g_j(u_j)$ is a sum over $n_j$ terms.
In this case, we could use existing results on the error
of Laplace approximations to one-dimensional integrals to show
 that the error in the first-order
Laplace approximation to the integral
is $O(\sum_{j=1}^{d} n_j^{-1})$. We show that a version
of this result also holds more generally, and find
similar expressions for the error in
higher-order Laplace approximations.
In Section \ref{sec:glmm}, we apply these results
to study the quality of Laplace approximations of
the likelihood for some generalized linear mixed models,
including a multilevel random intercept model
with any number of levels of hierarchy.

\section{Error in the log-integral approximation}
\label{sec:main_result}
\subsection{A series expansion for the log-integral}
\cite{Shun1995} give a series expansion for the log-integral
$\ell= \log L$. We use their expansion here, 
expressed with slightly different notation.
We write
\begin{equation}
  \label{eqn:logintegral_series}
  \ell = \tilde \ell_1 + \sum_{l=1}^\infty e_l,
\end{equation}
where $\tilde \ell_1$
is the first-order Laplace approximation to the log-integral,
and $e_l$ are contributions to the error in this
approximation of size decreasing with $l$, which we define
in Section \ref{sec:e_l_expansion}.

The first-order Laplace approximation to the log-integral is
\[\tilde \ell_1 = -\frac{1}{2} \log \det (g^{(2)}) + \frac{d}{2} \log(2 \pi) - g({\hat u})\]
where 
${\hat u} = \argmin_{{u} \in \mathbb{R}^d}\left\{g({u})\right\}$
and $g^{(2)} = g^{\prime \prime}({\hat u})$ is the matrix of second
derivatives of $g(.)$ with respect
to $u$, evaluated at ${\hat u}$.

Based on the decomposition \eqref{eqn:logintegral_series},
 we may also define an
order-$k$ Laplace approximation to the log-integral, for
$k \geq 2$, as
\[\tilde \ell_k = \tilde \ell_1 + \sum_{l=1}^{k-1} e_l.\]
What is meant by the order of a Laplace approximation
is not standard across the literature: our definition
is made by grouping together terms in a series expansion
to the log-integral in terms of their asymptotic order.
This is a different notion of order than
that used by \cite{Raudenbush2000}, who group
together terms according to the number of derivatives
required to compute them.

In this paper, we study the errors in these Laplace approximations
to the log-integral
\[\epsilon_k = \tilde \ell_k - \ell = - \sum_{l=k}^{\infty} e_k.\]

\subsection{An expansion over bipartitions}

\cite{Shun1995} give a series expansion
for the log-integral
in terms of particular bipartitions.
For positive integers $v$ and $m$,
define the set of $M$-bipartitions $\mathcal{M}_{v, m}$
to be all $(P, Q)$ such that 
$P = (p_1|\ldots|p_v)$ and $Q = (q_1|\ldots|q_m)$ 
are both partitions
of $\{1, \ldots, 2 m\}$, 
such that each block of $P$ contains at least three
elements and each block of $Q$ contains exactly two
elements.

For each $(P, Q) \in \mathcal{M}_{v, m}$, define
a corresponding graph $\mathcal{G}(P, Q)$ with
 vertices $1, \ldots, 2m$,  and an edge
between each pair of vertices contained
in the same block of either $P$ or $Q$.
If $\mathcal{G}_{P, Q}$ is a connected graph,
say that $(P, Q)$ is a connected bipartition,
and write $(P, Q) \in \mathcal{M}^C_{v, m}$.
We define the level of $(P, Q) \in \mathcal{M}_{v, m}$
to be $l = m - v$, and write $\mathcal{M}^C_{l}$
for all connected level-$l$ $M$-bipartitions.

\label{sec:e_l_expansion}

For a vector of indices $I$, write
$g_I({u}) = \nabla_{{u}_I} g({u})$
and $g_I = g_I({\hat u})$.
Let $g^{(k)}$ be the $k$-dimensional array with
entries $g^{(k)}_{j_1, \ldots, j_k} = g_{j_1, \ldots, j_k}$, and
write
$g^{jk} = \left(g^{(2)}\right)^{-1}_{jk}$.
Then define
\begin{equation}
  \label{eqn:e_P_Q}
  e_{P, Q} =
  \frac{(-1)^v}{(2m)!} \sum_{{j} \in [1:d]^{2m}}
  g_{{j}_{p_1}} \ldots  g_{{j}_{p_v}}
  g^{{j}_{q_1}} \ldots g^{{j}_{q_m}},
\end{equation}
where $[1:d]^{2m} = \{(j_1, \ldots, j_{2m}): j_l \in \{1, \ldots, d\}\},$
and ${j}_p$ is the sub-vector of ${j}=(j_1, \ldots, j_{2m})$ corresponding to the indices in $p$.

We may write the level-$l$ contribution
to the log-integral $e_l$ as a sum
of contributions from each connected level-$l$ $M$-bipartition, as
\begin{equation}
e_l = \sum_{(P, Q) \in \mathcal{M}_l^C} e_{P, Q}.
\end{equation}

\subsection{The level-$1$ contribution}
To demonstrate the definitions in Section \ref{sec:e_l_expansion},
we find the level-$1$ contribution $e_1$, used in
the second-order Laplace approximation.

There are three types of bipartitions in $\mathcal{M}^C_1$:
$(P_1, Q_1)$, where $P_1 = (1 \, 2 \, 3 \, 4)$ and
$Q_1 = (1 \, 2 \, | \, 3 \, 4)$;
$(P_2, Q_2)$, where $P_2 = (1 \, 2 \, 3 \, | \, 4 \, 5 \, 6)$
and $Q_2 = (1 \, 2 \, | \, 3 \, 4 \, | \, 5 \, 6)$;
and $(P_3, Q_3)$
where $P_3 = P_2$ and
$Q_3 = (1 \, 4 \, | \, 2 \, 5 \, | \, 3 \, 6)$.
While there are other bipartitions in $\mathcal{M}^C_1$,
they are all similar to one of these three, in that
they may be obtained by rearranging the labels 
$\{1, \ldots, 2m\}$, and so give the same contribution
$e_{P, Q}$. For example, the bipartition
$P_1^* = (1 \, 2 \, 3 \, 4)$, $Q_1^* = (1 \, 3 \, | \, 2 \, 4)$
 may be obtained from $(P_1, Q_1)$
 by exchanging $2$ and $3$, and $e_{P_1^*, Q_1^*} = e_{P_1, Q_1}$.
From \eqref{eqn:e_P_Q}, we have
\begin{align}
  \label{eqn:e_P_Q_for_e_1}
  e_{P_1, Q_1} &= -\frac{1}{4!} \sum_{j_1, \ldots, j_4} g_{j_1 j_2 j_3 j_4} g^{j_1 j_2} g^{j_3 j_4} \notag \\
  e_{P_2, Q_2} &=  \frac{1}{6!} \sum_{j_1, \ldots, j_6} g_{j_1 j_2 j_3} g_{j_4 j_5 j_6} g^{j_1 j_2} g^{j_3 j_4} g^{j_5 j_6} \\
  e_{P_3, Q_3} &= \frac{1}{6!} \sum_{j_1, \ldots, j_6} g_{j_1 j_2 j_3} g_{j_4 j_5 j_6} g^{j_1 j_4} g^{j_2 j_5} g^{j_3 j_6}. \notag
\end{align}

\cite{McCullagh1987} lists
$4$ bipartitions similar to $(P_1, Q_2)$, $9$
similar to $(P_2, Q_2)$ and $6$ similar to $(P_3, Q_3)$,
so the level-$1$ contribution is
$e_1 = 3 e_{P_1, Q_1} + 9 e_{P_2, Q_2} + 6 e_{P_3, Q_3},$
and the second-order Laplace approximation to the log-likelihood is
$\tilde \ell_2 = \tilde \ell_1 + e_1.$

There may be more efficient ways to compute
$e_1$ than direct computation of the  sums in
\eqref{eqn:e_P_Q_for_e_1}.
For example, \cite{Zipunnikov2011} describe a more efficient
method for computing these terms for a generalized linear mixed model.

\subsection{Asymptotic order of terms}

Write
$a = \Theta(b)$ if $a = O(b)$ and $a^{-1} = O(b^{-1})$,
so $a$ grows at the same rate as $b$.
For a random variable $A$, write
$A = \Theta_p(b_n)$ if $A = O_p(b_n)$ and $A^{-1} = O_p(b_n^{-1})$.

We use a particular notion of a random array being order $1$
in probability. Suppose $A$ is a $k$-dimensional array,
  with entries $A_{j_1, \ldots, j_k}$
  for each $j_i \in \{1, \ldots, d\}$.
  If $k = 1$, say $A = O_p^*(1)$
  if $A_j = O_p(1)$ for each $j = 1, \ldots, d$.
  If $k \geq 2$, let
  \[A^i_{j} = \sum_{j_1 = 1}^d  \ldots \sum_{j_{i-1} = 1}^d \sum_{j_{i+1} = 1}^d \ldots \sum_{j_k=1}^d
  \vert A_{j_1, \, \ldots,\, j_{i-1}, \,j, \, j_{i+1}, \, \ldots,\, j_k}\vert,\]
  and say $A = O_p^*(1)$ if $A^i_{j} = O_p(1)$ for each $i = 1, \ldots, k$
  and $j = 1, \ldots, d$.

If $A$ is a diagonal array, then $A = O_p^*(1)$
if the diagonal
entries $A_{j, j, \ldots, j} = O_p(1)$.

\subsection{Assumptions}

We assume that $g(.)$ in \eqref{eq:form_of_integral}
satisfies some conditions.

\begin{condition}
  \label{cond:g_smooth_unique_minimum}
  $g(.)$ is a smooth function with a unique minimum.
\end{condition}

For a given choice of normalizing terms
$n_1, \ldots, n_d$, and for each
vector of indices $I$, define the normalized derivatives
\[f_I = g_I \prod_{j \in I} n_j^{-1/|I|} ,\]
and write $f^{(k)}$ for the $k$-dimensional array with
entries $f^{(k)}_{j_1, \ldots, j_k} = f_{j_1, \ldots, j_k}$.
We write $f^{jk} = [(f^{(2)})^{-1}]_{jk}$.

\begin{condition}
  \label{cond:f_O*1}
  There is some choice of normalizing terms
$n_1, \ldots, n_d$ such that the
  normalized derivative arrays $f^{(k)}$ satisfy
 $f^{(k)} = O_p^*(1)$ for all $k \geq 3$, and
  $[f^{(2)}]^{-1} = O_p^*(1)$.
\end{condition}

The normalizing terms are often chosen so that
$g_{jj} = \Theta_p(n_j)$, and we may think of $n_j$
as an effective sample size for $u_j$.

\subsection{Error in log-integral approximations}
We state here our main result, which is proved
in Appendix \ref{sec:proofs_main}. 
\begin{theorem}
  \label{thm:error_in_logintegral}
  Suppose $L$ is of form \eqref{eq:form_of_integral},
  where $g(.)$ satisfies Conditions \ref{cond:g_smooth_unique_minimum}
  and \ref{cond:f_O*1}, for some choice of normalizing
  terms $n_1, \ldots, n_d$. Then
  the error in the order-$k$ Laplace approximation to $\log L$ is
$\epsilon_k = O_p(\sum_{j=1}^d n_j^{-k}).$
\end{theorem}

\subsection{Linear reparameterizations}
Laplace approximations are invariant
to linear reparameterizations. That is, if 
${v} = A {u}$, where $A$ is an invertible $d \times d$ matrix,
then writing $g_v(v) = g(A^{-1} v) + \log \det (A),$
and
\[L^{(v)} =  \int_{\mathbb{R}^d} \exp\{-g_v({v})\} d{v},\]
we have $L^{(v)} = L$, and 
the order-$k$ Laplace
approximation of $L$
is unchanged by the reparameterization, so that
$\tilde L^{(v)}_k = \tilde L_k.$ 

In many situations, Condition \ref{cond:f_O*1} does
not hold in the original parameterization, but does
hold after making a suitable linear reparameterization,
so we may still apply Theorem \ref{thm:error_in_logintegral}.
We give an example of this in Section \ref{sec:multilevel}.

\section{Application to likelihood approximation for generalized linear mixed models}
\label{sec:glmm}
\subsection{The model}
\label{sec:glmm_model}
In a generalized linear mixed model, the
distribution of the
response ${Y} = (Y_1, \ldots, Y_n)$ is determined
by a linear predictor ${\eta} = (\eta_1, \ldots, \eta_n)$.
Conditional on $\eta$, the components $Y_i$
of the response are independent, with known density function
$f(y_i | \eta_i)$.
We assume an exponential family with canonical link, so that
\[\log f(y_i | \eta_i) = \frac{y_i \eta_i - b(\eta_i)}{a_i(\phi)},\]
where $b(.)$ is a smooth and convex function,
$a_i(\phi) > 0$, and $\phi$ is the
dispersion parameter, which we assume here to be known.
The linear predictor is modelled as
 ${\eta} = X {\beta} + Z {u},$ where $X \in \mathbb{R}^{n \times p}$ and
 $Z \in \mathbb{R}^{n \times d}$ are
design matrices, ${\beta} \in \mathbb{R}^p$ is a 
vector of fixed effects, and ${u} \in \mathbb{R}^d$ is a vector
of random effects. We assume
that ${u} \sim N_d(0, \Sigma({\psi}))$, where ${\psi} \in \mathbb{R}^{q}$
is an unknown parameter, and write ${\theta} = ({\beta}, {\psi})$ for
the full vector of unknown parameters.

\subsection{The likelihood}
The likelihood for this model is
\begin{equation}
\label{eqn:GLMM_likelihood}
L({\theta}) = \int_{\mathbb{R}^d}  \exp\{ - g({u}; {\theta}) \} d{u},
\end{equation}
where
\begin{equation}
  \label{eqn:GLMM_g}
  g({u}; {\theta}) = h({u}; {\beta}) - \log \phi_d({u}; 0, \Sigma({\psi})),\
  \end{equation}
\begin{align}
  \label{eqn:GLMM_h}
  h({u}; {\beta}) &= \sum_{i=1}^n - \log f(y_i | \eta_i = {X_i}^T {\beta} + {Z_i}^T {u}) \notag \\
  &= \sum_{i=1}^n \frac{b({X_i}^T {\beta} + {Z_i}^T {u}) - y_i({X_i}^T {\beta} + {Z_i}^T {u})}{a_i(\phi)}
 \end{align}
and $\phi_d(.; {\mu}, \Sigma)$ is the $N_d({\mu}, \Sigma)$
density function.
The $d$-dimensional integral in \eqref{eqn:GLMM_likelihood}
is typically intractable, except in the special
case of a linear mixed model where $Y_i|\eta_i$ are
normally distributed. Because of this intractability,
it is common to use some numerical approximation $\tilde L({\theta})$
to the likelihood, and first-order Laplace approximation
is often used. For example, by default
the \texttt{lme4} R package
\citep{lme4} uses a first-order Laplace approximation
to the likelihood for inference, and the integrated
nested Laplace approximations of \cite{Rue2009}
is a Bayesian approach based on a Laplace
approximation to the likelihood.

\subsection{Assumption checking}

In order to apply Theorem \ref{thm:error_in_logintegral} to
the likelihood of a generalized linear mixed
model, we will first have to show
that $g(.)$ as defined in \eqref{eqn:GLMM_g} satisfies Conditions
\ref{cond:g_smooth_unique_minimum} and \ref{cond:f_O*1}.
We drop $\theta$ from the notation, so that \eqref{eqn:GLMM_likelihood}
is of form \eqref{eq:form_of_integral}.

We can show Condition \ref{cond:g_smooth_unique_minimum}
holds in all cases. The proof is in Appendix \ref{sec:proofs_examples}.
\begin{proposition}
  \label{proposition:GLMM_g_convex}
  Let $g({u})$ be as defined in \eqref{eqn:GLMM_g},
  where $\Sigma$ is a positive definite matrix. 
  Then $g(.)$ satisfies Condition \ref{cond:g_smooth_unique_minimum}.
\end{proposition}

We need to show that Condition \ref{cond:f_O*1} holds on a
case-by-case basis.
In our examples, we choose the normalizing term $n_j$
to be the number of observations which involve
$u_j$.

\subsection{A two-level random intercept model}
\label{sec:two_level_ri}
We consider a two-level random intercept model, which is a special case of
the generalized linear mixed
model of Section \ref{sec:glmm_model} in which each observation $i$
is contained in a cluster $c(i)$. Observations in the same
cluster $j$ are correlated by a shared random effect $u_j$.
The linear predictor is
$\eta_{i} = {x_i}^T {\beta} + u_{c(i)}  \, (i = 1, \ldots, n),$
where we suppose the
$u_j$ are independent $N(0, \sigma^2)$ random variables.
In the notation of Section \ref{sec:glmm_model}, we have
$Z_{i, c(i)} = 1$, and $Z_{i, j} = 0$ if $j \not = c(i)$ and
$\Sigma = \sigma^2 I$, where $I$ is an
identity matrix.

In this special case, the likelihood \eqref{eqn:GLMM_likelihood}
simplifies into a product of one-dimensional integrals
\[L({\theta}) = \prod_{j=1}^d \int \prod_{i: c(i) = j} f(y_i | \eta_i = {x_i}^T {\beta} + u_j) \phi(u_j; 0, \sigma^2) du_j.\]
The log-likelihood may be written as a sum
\begin{equation}
  \label{eq:loglikelihood_two_level}
  \ell({\theta}) = \sum_{j=1}^d \log \int \prod_{i: c(i) = j} f(y_i | \eta_i = {x_i}^T {\beta} + u_j) \phi(u_j; 0, \sigma^2) du_j,
  \end{equation}
so $\epsilon_k$ is a sum of separate error terms.

\begin{proposition}
  \label{proposition:two_level}
Suppose we have a two-level random
intercept model,
with $n_j$ observations on cluster $j$, for $j = 1, \ldots d$.
The error in the order-$k$ Laplace
approximation to the log-likelihood is
\[\epsilon_k({\theta}) =
\tilde \ell_k({\theta}) - \ell({\theta}) = O_p\Big(\sum_{j=1}^d n_j^{-k}\Big).\]
\end{proposition}
\begin{proof}
The derivative arrays $g^{(k)}$ are diagonal
for all $k$, with diagonal entries
$g_{j \ldots j} = \Theta_p(n_j)$, so Condition \ref{cond:f_O*1} holds
with normalizing terms $n_1, \ldots, n_d$.
Theorem \ref{thm:error_in_logintegral}
gives that $\epsilon_k({\theta}) = O_p(\sum_{j=1}^d n_j^{-k})$,
as required.
\end{proof}

In the balanced case, where all $n_j = n d^{-1}$,
$\epsilon_k = O_p(d^{k + 1} n^{-k})$. This tends to zero as $n \rightarrow \infty$
if $d = o(n^{(k+1)/k})$. The error in the first-order Laplace approximation
tends to zero if $d = o(n^{1/2})$.

In an unbalanced case, the result can be quite different. As an
extreme example, suppose
\[n_j = \begin{cases}
  \log d & \text{if $j = 1, \ldots, d-1$} \\
  n - (d-1) \log d & \text{if $j = d$,}
\end{cases}\]
where $n > d \log d$.
 Then
\[\epsilon_1 = O_p\big((d-1) (\log d)^{-1} + (n - (d-1) \log d)^{-1} \big) = O_p\big(d (\log d)^{-1}\big),\]
which tends to infinity as $d \rightarrow \infty$, now matter how
large $n$ is relative to $d$. For example, if $n = d^4$, then $d = o(n^{1/3})$,
but $\epsilon_1 \rightarrow \infty$.

 \subsection{A multilevel random intercept model}
 \label{sec:multilevel}
Suppose that each observation $i$
is contained in a level-2 cluster $c_2(i)$,
and that each level-2 cluster $j$ is itself contained within
a hierarchy of higher-level clusters,
$c_l(j)$, $j = 3, \ldots, L$. The clusters are nested
within one another, so that if $c_l(j) = c_l(k)$, then
$c_{l+1}(j) = c_{l+1}(k)$.
The linear predictor is
\[\eta_{i} = {x_i}^T {\beta} + u^{(2)}_{c_2(i)} + \sum_{l=3}^L u^{(l)}_{c_l(c_2(i))} \quad
(i = 1, \ldots, n),\]
where we assume
$u^{(l)}_j \sim N(0, \sigma_l^2)$, $l = 2, \ldots, L$,
with all the $u_j$ independent.
Suppose that there are $d$ level-2 clusters in total,
and $d_l$ level-$l$ clusters, for each $l = 3, \ldots, L$.
 It is no
longer possible to write the log-likelihood
as a sum of one-dimension log-integrals
as in \eqref{eq:loglikelihood_two_level}. Since an accurate approximation
to the exact log-likelihood is no longer readily available,
it is important to understand the quality of the Laplace approximation
in this case.

Condition \ref{cond:f_O*1} does not hold for this parameterization,
so we define a new parameterization of the model.
Let $v_j = u^{(2)}_j + \sum_{l=3}^L u^{(l)}_{c_l(j)}$ for $j = 1, \ldots, d$.
We have
$\eta_{i} = {x_i}^T {\beta} + v_{c_2(i)}$, where there are now
a total of $d$ random effects, rather
than $d + d_3 + \ldots + d_L$ in the original parameterization.
We have reduced the structure to the two-level random
intercept model of Section  \ref{sec:two_level_ri}, except now
${v} \sim N_d(0, \Sigma)$, where
\[\Sigma_{jk} = \begin{cases}
  \sigma_2^2 + \sigma_3^2 + \ldots + \sigma_L^2 & \text{if $j = k$} \\
  \sigma_3^2 + \ldots + \sigma_L^2  & \text{if $j \not = k$, but $c_3(j) = c_3(k)$} \\
 \vdots & \vdots \\
  \sigma_l^2 + \ldots + \sigma_L^2 & \text{if $c_{l-1}(j) \not = c_{l-1}(k)$, but $c_{l}(j)  = c_{l}(k)$} \\
 \vdots & \vdots \\
    \sigma_L^2 & \text{if $c_{L-1}(j) \not = c_{L-1}(k)$, but $c_{L}(j)  = c_{L}(k)$} \\
   0 & \text{if $c_{L}(j) \not = c_{L}(k).$}
    \end{cases}\]

\begin{proposition}
  \label{proposition:multilevel}
  Suppose we have an $L$-level random
intercept model with independent random effects,
with $n_j$ observations in level-$2$ cluster $j$, for $j = 1, \ldots d$.
The error in the order-$k$ Laplace
approximation to the log-likelihood is
\[\epsilon_k({\theta}) =
\tilde \ell_k({\theta}) - \ell({\theta}) = O_p\Big(\sum_{j=1}^d n_j^{-k}\Big).\]
\end{proposition}
The proof is in Appendix \ref{sec:proofs_examples}.

The asymptotic order of the error in a Laplace
approximation to the log-likelihood depends
on the number of observations in each of the level-$2$
clusters, but not on how these level-$2$ clusters
are grouped into higher-level clusters.

\subsection{Impact on approximate likelihood inference}
When an approximate likelihood $\tilde L({\theta})$ is used for inference,
the impact of the error in the likelihood approximation
on the resulting inference is of more interest than the size of
that error itself. If the error in the log-likelihood
$\epsilon({\theta}) = \log \tilde L({\theta}) - \log L({\theta})$
tends to zero in probability, uniformly in $\theta$,
\cite{Douc2004} show that the approximate likelihood estimator
$\tilde \theta$ will be fully efficient, and have the same
first-order asymptotic distribution as the maximum likelihood estimator.
In our examples, if
\begin{equation}
  \label{eqn:efficiency_condition}
  \sum_{j=1}^d n_j^{-k} \rightarrow 0 \text{ as $d \rightarrow \infty$}
  \end{equation}
we expect the order-$k$ Laplace estimator to be fully efficient.
In order to make the argument rigorous,
we would need to show that the supremum of the error in log-likelihood
in some region around the true parameter value tends to zero.

However, condition \eqref{eqn:efficiency_condition} is likely to be
stronger than necessary for a order-$k$ Laplace
estimator to be fully efficient. 
\cite{Ogden2017}
gives conditions on the size of the error in the score function
$\nabla_\theta \epsilon(\theta)$
which ensure that inference with an approximate likelihood
retains the same first-order properties as inference with
the exact likelihood.
By studying this error in score,
it should be possible to show
that the order-$k$ Laplace estimator is fully efficient
under a weaker condition than \eqref{eqn:efficiency_condition}.
Some modification of the results of \cite{Ogden2017} would be required before they
could be used in this case, as
information on different components of the
parameter vector may grow at different rates \citep{Nie2007}.

\begin{appendices}

\section{Proof of main result}
\label{sec:proofs_main}

To prove Theorem \ref{thm:error_in_logintegral}, we aim to find
the size of the contribution from
each bipartition $(P, Q)$.
\begin{lemma}
\label{lemma:e_P_Q}
  Suppose Condition \ref{cond:f_O*1} holds.
  For each fixed bipartition $(P, Q) \in \mathcal{M}^C_{l}$
\[e_{P, Q} =
 O_p\Big(\sum_{j=1}^d n_j^{-l}\Big).\]
\end{lemma}

Given Lemma \ref{lemma:e_P_Q}, the proof of Theorem
\ref{thm:error_in_logintegral} is straightforward:
\begin{proof}[Proof of Theorem \ref{thm:error_in_logintegral}]
  By Lemma \ref{lemma:e_P_Q}, we have
  $e_{P, Q} = O_p(\sum_{j=1}^d n_j^{-l})$
  for each fixed bipartition $(P, Q) \in \mathcal{M}^C_{l}$.
  Combining the contributions from each bipartition in $\mathcal{M}^C_{l}$,
  we have $e_l =  O_p(\sum_{j=1}^d n_j^{-l})$,
so
 $\epsilon_k =  - \sum_{l=k}^{\infty} e_k = O_p(\sum_{j=1}^d n_j^{-k})$,
as required.
\end{proof}

In order to prove Lemma \ref{lemma:e_P_Q}, we need some auxiliary
results.

  \begin{proposition}
    \label{proposition:polynomial_e_P_Q}
    Let $(P, Q)$ be a fixed $(v, 2m)$ bipartition.
    For each $j = (j_1, \ldots, j_{2m}) \in [1:d]^{2m}$, write  
    $A_{P, Q}(j) =  f_{j_{p_1}} \ldots f_{j_{p_v}}
   f^{j_{q_1}} \ldots f^{j_{q_m}}.$
 Then
 \[
 e_{P, Q} =   \frac{(-1)^v}{(2m)!} \sum_{j \in [1:d]^{2m}} n_{j_1}^{c_1} \ldots n_{j_{2m}}^{c_{2m}} A_{P, Q}(j)
\]
where $\sum_{j=1}^{2m} c_j = -l$, and each $c_j < 0$.
  \end{proposition}
  \begin{proof}
    We may write
    \begin{align*}
      e_{P, Q} &=   \frac{(-1)^v}{(2m)!} \sum_{{j} \in [1:d]^{2m}}
      \prod_{p \in P} \prod_{k \in {j}_p} n_k^{1/|p|} f_{{j}_{p}}
       \prod_{q \in Q} \prod_{l \in {j}_q} n_l^{-1/2} f^{{j}_{q}}
      \\
      &=   \frac{(-1)^v}{(2m)!} \sum_{{j} \in [1:d]^{2m}} n_{j_1}^{c_1} \ldots n_{j_{2m}}^{c_{2m}}
      \prod_{p \in P}  f_{{j}_{p}}
      \prod_{q \in Q} f^{{j}_{q}}
    \end{align*}
    for some $c_1, \ldots c_{2m}$.
    We have $c_i = -\frac{1}{2} + \frac{1}{|p|}$, for whichever $p$ contains
    $i$, so $c_i < 0$ as $|p| \geq 3$. We have
    \[\sum_{i=1}^{2m} c_i = -m + \sum_{p \in P} \sum_{i \in p} \frac{1}{|p|}
    = -m + \sum_{p \in P} 1 = -m + v = -l\]
    which gives the result.
  \end{proof}

\begin{proposition}
  \label{proposition:product_O*1}
  Suppose $A = O_p^*(1)$ and $B = O_p^*(1)$, and
  $C$ is the $k$-dimensional array with entries
  $C_{j} = A_{j_{S}} B_{j_{T}}$,
  where $j = (j_1, \ldots, j_k)$, and $S, T \subseteq \{1, \ldots, k\}$,
  such that $S \cup T = \{1, \ldots, k\}.$
  If $S \cap T \not = \emptyset$, then
  $C = O_p^*(1)$.
  \end{proposition}
\begin{proof}

  We proceed by induction on $k = \dim(C) = |S \cup T|$.

  In the case $k = 1$, we have $S = T$, since $S \cap T \not = \emptyset$.
  So $C_{j_1} = A_{j_1} B_{j_1} = O_p(1)$, so $C = O_p^*(1)$.

  Now we suppose the hypothesis is true for $\dim(C) = k - 1$,
  and consider $\dim(C) = k \geq 2$.

  We have \[C^i_{j_i} = \sum_{j_l: l \not = i} |C_{j}|
  = \sum_{j_l: l \not = i, a} \sum_{j_{a}} |C_{j}|.\]
    Writing $j_{-a} = (j_1, \ldots, j_{a-1}, j_{a+1}, \ldots, j_k)$
    and
    $C^{-a}_{j_{-a}} = \sum_{j_a} |C_{j}|$,
  if we can show that $C^{-a} = O_p^*(1)$
  for some $a \not = i$, then
  \[C^i_{j_i} = \sum_{j_l: l \not = i, a} C^{-a}_{j_{-a}} = O_p(1),\]
  so that $C = O_p^*(1)$.

  $C^{-a}$ has entries
  \begin{equation*}
C^{-a}_{j_{-a}} = \sum_{j_a}  |A_{j_S} B_{j_T}|
  = \begin{cases}
    |B_{j_T}|  \sum_{j_a}  |A_{j_S}|  & \text{if $a \in S$, $a \not \in T$} \\
    |A_{j_S}| \sum_{j_a}  |B_{j_T}| & \text{if $a \not \in S$, $a \in T$} \\
    \sum_{j_a}  |A_{j_S} B_{j_T}| & \text{if $a \in S$, $a \in T$ }
  \end{cases}
  \end{equation*}
  In the first case, we must have $\dim(A) \geq 2$, otherwise
  $S = \{a\}$ and $S \cap T = \emptyset$, which would be a contradiction.
  Since $A = O_p^*(1)$, the array $A^{-a}$ with entries
  $A^{-a}_{j_{S \setminus a}} = \sum_{j_a}  |A_{j_S}|$ must also be $O_p^*(1)$.
  So the array $C^{-a}$
  with entries $C^{-a}_{j_{-a}} = |B_{j_T}| A^{-a}_{j_{S \setminus a}}$ is
  $O_p^*(1)$,
  by the induction hypothesis, since $\dim(C^{-a}) = k - 1$.
  Similarly, in the second case $C^{-a} = O_p^*(1)$.
  In the third case,
\[    C^{-a}_{j_{-a}} = \sum_{j_a}  |A_{j_S} B_{j_T}|
    \leq  \sum_{j_a} |A_{j_S}|  \sum_{j_a}  |B_{j_T}|
    = A^{-a}_{j_{S \setminus a}} B^{-a}_{j_{T \setminus a}} \]
    by the Cauchy--Schwarz inequality.
    So $C^{-a} = O_p^*(1)$, by the induction hypothesis.

    In all cases $C^{-a} = O_p^*(1)$, so $C = O_p^*(1)$, as required.
\end{proof}
    
  \begin{proposition}
    \label{proposition:A_Op*1}
    Suppose Condition \ref{cond:f_O*1} holds.
     Let $(P, Q)$ be a fixed $(v, 2m)$ bipartition.
    Then $A_{P, Q} = O_p^*(1)$.
  \end{proposition}

  \begin{proof}
    We have 
    $A_{P, Q}({j}) = \prod_{p \in P} f_{{j}_p} \prod_{q \in Q} f^{{j}_q}.$
    Since $f^{(k)} = O_p^*(1)$ for $k \geq 3$ and $[f^{(2)}]^{-1} = O_p^*(1)$,
    $A_{P, Q}$ is a product of $O_p^*(1)$ arrays.

    We build up this product one term at a time, at each step applying
    Proposition \ref{proposition:product_O*1} to show that the product remains $O_p^*(1)$.
    
    We start with an arbitrary $p_1 \in P$,
    and choose $q_1 \in Q$ such that one element of $q_1$ is in $p_1$,
    and the other is not in $p_1$, and therefore must be in some other block $p_2 \in P$.
    If $v > 1$, it will always be possible to find such a $q$,
    because $(P, Q)$ is a connected
    bipartition, so
    blocks of $P$ (which form disjoint clusters in $\mathcal{G}_{P, Q}$) 
    are connected by blocks of $Q$.
    
    Let $S^1$ be the array with entries
    $S^1_{{j}_{p_1 \cup q_1}} = f_{{j}_{p_1}} f^{{j}_{q_1}}.$ Then 
    $S^1 = O_p^*(1)$ by Proposition \ref{proposition:product_O*1},
    as $p_1 \cap q_1 \not = \emptyset$.
    Let $T^1$ be the array with entries
    $T^1_{{j}_{p_1 \cup p_2}} =  f_{{j}_{p_2}} S^1_{{j}_{p_1 \cup q_1}}.$ Then
    $T^1 = O_p^*(1)$, as $p_2 \cap (p_1 \cup q_1) = p_2 \cap q_1 \not = \emptyset$.

    We continue to choose alternating terms from blocks of $Q$ and $P$,
    at step $k$ choosing a block $q_k$ with one entry in $p_1 \cup \ldots \cup p_k$,
    and the other entry in a new block $p_{k+1}$. At each stage $k$ we have
    $S^k_{{j}_{p_1 \cup \ldots \cup p_{k} \cup q_k}} = f^{{j}_{q_k}} \, T^{k-1}_{{j}_{p_1 \cup \ldots \cup p_{k}}}$
    and
    $T^{k}_{{j}_{p_1 \cup \ldots \cup p_k \cup p_{k+1}}} =  f_{{j}_{p_{k+1}}} \, S^{k}_{{j}_{p_1 \cup \ldots \cup p_k \cup q_k}},$
    where $S^{k} = O_p^*(1)$ and $T^{k} = O_p^*(1)$.
    
    We continue until we have included all blocks of $P$, and have 
    $T^{v-1}_{{j}_{p_1 \cup p_2 \cup \ldots \cup p_v}} = T^{v-1}_{{j}}$
    where $T^{v - 1} = O_p^*(1)$. We have already included terms
    from $v - 1$ blocks of $Q$. We may multiply in the remaining $2m - v + 1$
    blocks of $Q$ while retaining an $O_p^*(1)$ array by Proposition \ref{proposition:product_O*1}
    as $T^{v-1}$ is an array on all indices $j_1, \ldots, j_{2m}$, and
    $q \cap (1:2m)  = q \not = \emptyset$
    for each $q \in Q$.
    So $A_{P, Q} = O_p^*(1)$, as required.
  \end{proof}
  
\begin{proof}[Proof of Lemma \ref{lemma:e_P_Q}]
  By Proposition \ref{proposition:polynomial_e_P_Q}
 \begin{align}
   \vert e_{P, Q} \vert &=   \frac{1}{(2m)!}\, \Bigg\vert \sum_{{j} \in [1:d]^{2m}} n_{j_1}^{c_1} \ldots n_{j_{2m}}^{c_{2m}} A_{P, Q}({j}) \, \Bigg\vert \notag \\
   &\leq   \frac{1}{(2m)!} \sum_{{j} \in [1:d]^{2m}} n_{j_1}^{c_1} \ldots n_{j_{2m}}^{c_{2m}} \vert A_{P, Q}({j}) \vert.
   \label{eqn:abs_e_P_Q}
 \end{align}
 We apply the weighted form of the inequality of arithmetic and geometric means,
 which states that given non-negative numbers $x_1, \ldots, x_{n}$ and non-negative
 weights $w_1, \ldots, w_{n}$ with $\sum_{i} w_i = 1$,
 \[\prod_{i=1}^n x_i^{w_i} \leq \sum_{i=1}^n w_i x_i .\]
 Here, we let $n = 2m$, $x_i = n_{j_i}^{-l}$ and $w_i = -c_i / l$, to give that
 \begin{equation}
   n_{j_1}^{c_1} \ldots n_{j_{2m}}^{c_{2m}} \leq \sum_{i=1}^{2m} w_i n_{j_i}^{-l}.
   \label{eqn:am_gm_bound}
   \end{equation}
 Putting \eqref{eqn:am_gm_bound} back into \eqref{eqn:abs_e_P_Q} gives
\begin{align*}
  \vert e_{P, Q} \vert &\leq  \frac{1}{(2m)!} \sum_{{j} \in [1:d]^{2m}} \sum_{i=1}^{2m} w_i n_{j_i}^{-l} \,|A_{P, Q}({j})| \\
  &= \frac{1}{(2m)!} \sum_{i=1}^{2m} \sum_{j_i = 1}^d  w_i n_{j_i}^{-l} \sum_{j_1, \ldots, j_{i-1}, j_{i+1}, \ldots, j_{2m}} |A_{P, Q}({j})| \\
  &=  \frac{1}{(2m)!} \sum_{i=1}^{2m} \sum_{j_i = 1}^d w_i n_{j_i}^{-1} A_{j_i}^{i}\\
  &=  \frac{1}{(2m)!} \sum_{i=1}^{2m} O_p\Big( \sum_{j_i = 1}^d n_{j_i} ^{-l}\Big)
  = O_p\Big(\sum_{j = 1}^d n_{j} ^{-l}\Big)
\end{align*}
since $m$ is fixed as $d \rightarrow \infty$.
\end{proof}

\section{Proofs for examples}
\label{sec:proofs_examples}
\begin{proof}[Proof of Proposition \ref{proposition:GLMM_g_convex}]
  The matrix of second derivatives of $g(.)$ with respect
  to ${u}$ is
  $g^{(2)}({u}) = h^{(2)}({u}) + \Sigma^{-1},$
  where $h^{(2)}({u})$ is the matrix of second derivatives
  of $h(.)$ with respect to ${u}$, and $h(.)$
  is defined in \eqref{eqn:GLMM_h}.
  We have
  $h^{(2)}({u}) = Z^T W({u}) Z,$
  where $W({u})$ is a diagonal matrix with diagonal entries
  \[W_{ii}({u}) = \frac{b^{\prime \prime}(X_i^T \beta + Z_i^T u)}{a_i(\phi)}.\]
  But $a_i(\phi) > 0$, and since $b(.)$ is a convex function $b^{\prime \prime}(X_i^T \beta + Z_i^T u) \geq 0$, so $W_{ii}(u) \geq 0$ for all $u$. So $W({u})$ is a non-negative
  definite matrix, and for any $x \in \mathbb{R}^d$,
  $x^T h^{(2)}({u}) x = (x Z)^T W (Z x) \geq 0,$
  which means that $h^{(2)}({u})$ is non-negative definite.
  Since $\Sigma^{-1}$ is positive definite, this means
  that $g^{(2)}({u})$ is positive definite for all ${u}$,
  so $g(.)$ is strictly convex, and therefore has a unique minimum.
  Since $b(.)$ is a smooth function, so is $g(.)$, so Condition
\ref{cond:g_smooth_unique_minimum} holds.
\end{proof}

\begin{proof}[Proof of Proposition \ref{proposition:multilevel}]

  To prove the result, we need to show that after reparameterization
  Condition \ref{cond:f_O*1} holds with normalizing terms $n_1, \ldots, n_d$,
  so that we can apply Theorem
  \ref{thm:error_in_logintegral}.
  
  For $k \geq 3$, $ g^{(k)}$ is diagonal with
  diagonal terms $g_{j \ldots j} = \Theta_p(n_j)$, so
  $ f^{(k)} = O_p^*(1)$ for $k \geq 3$.
 It remains to show that $[ f^{(2)}]^{-1} = O_p^*(1)$.
  
  Write
  \[\Sigma^{[l]}_{jk} = \begin{cases}
  \sigma_2^2 + \sigma_3^2 + \ldots + \sigma_l^2 & \text{if $j = k$} \\
  \sigma_3^2 + \ldots + \sigma_l^2  & \text{if $j \not = k$, but $c_3(j) = c_3(k)$} \\
 \vdots & \vdots \\
  \sigma_l^2 & \text{if $c_{l-1}(j) \not = c_{l-1}(k)$, but $c_{l}(j)  = c_{l}(k)$,}
  \end{cases}\]
  so that $\Sigma = \Sigma^{[L]}$.

  $\Sigma^{[l]}$ is a block-diagonal matrix, with $d_l$ blocks,
  one for each level-$l$ cluster. We have
  \[\Sigma^{[l]}_{jk} =
  \begin{cases}
    \Sigma^{[l-1]}_{jk} + \sigma_l^2 & \text{if $c_l(j) = c_l(k)$} \\
    0 & \text{otherwise}
  \end{cases}
  \]

   Write $\Sigma_{[l]}^{jk} = (\Sigma^{[l]})^{-1}_{jk}$.
  Applying the Sherman--Morrison formula to invert each block
  of $\Sigma^{[l]}$ gives
  \begin{equation}
    \label{eqn:Sigma_inv_recursion}
  \Sigma_{[l]}^{jk} =
    \begin{cases}
    \Sigma_{[l-1]}^{jk} - \frac{\sigma_l^2 r_j r_k}{1 + \sigma_l^2 s_{c_l(j)}} & \text{if $c_l(j) = c_l(k)$} \\
    0 & \text{otherwise,}
    \end{cases}
   \end{equation}
    where
    \begin{equation}
      \label{eqn:Sigma_inv_rowsums}
      r_j = \sum_{k: c_l(k) = c_l(j)} \Sigma_{[l-1]}^{jk}, \quad  s_c = \sum_{j: c_l(j) = c} r_j.
    \end{equation}

    We hypothesize that
    \begin{equation}
      \label{eqn:Sigma_inv_sizes}
    \Sigma_{[l]}^{jk} = 
    \begin{cases}
  \Theta(1) & \text{if $j = k$} \\
  \Theta\big((d^{3}_{c_3(j)})^{-1} \big) & \text{if $j \not = k$, but $c_3(j) = c_3(k)$} \\
 \vdots & \vdots \\
 \Theta\big((d^{l}_{c_l(j)})^{-1}\big) & \text{if $c_{l-1}(j) \not = c_{l-1}(k)$, but $c_{l}(j)  = c_{l}(k)$} \\
 0 & \text{otherwise,}
    \end{cases}
    \end{equation}
    and prove this by induction on $l$. This claim is true for $l = 2$, as
    $\Sigma_{[2]}^{-1} = \sigma_2^{-2} I$. For $l \geq 2$, applying the induction
    hypothesis to \eqref{eqn:Sigma_inv_rowsums}, we find
    $r_j = \Theta(1)$, so $s_c = \Theta(d_c^l)$ and
    \begin{equation}
      \label{eqn:Sigma_inv_increment}
      \frac{\sigma_l^2 r_j r_k}{1 + \sigma_l^2 s_{c_l(j)}} = \Theta\big((d_c^l)^{-1}\big).
      \end{equation}
    Substituting \eqref{eqn:Sigma_inv_increment} into \eqref{eqn:Sigma_inv_recursion}
    proves \eqref{eqn:Sigma_inv_sizes}.
    
    Now write $g^{[l]}_{jk} = h_{jk} - \Sigma_{[l]}^{jk},$
      so that $g_{jk} = g^{[L]}_{jk}$. 
      Again, $g^{[l]}$ is block-diagonal, and
      \begin{align*}
        g^{[l]}_{jk} &= g^{[l-1]}_{jk} - (\Sigma^{jk}_{[l]} - \Sigma^{jk}_{[l-1]}) \\
        &=  
 \begin{cases}
          g^{[l-1]}_{jk} + \frac{\sigma_l^2 r_j r_k}{1 + \sigma_l^2 s_{c_l(j)}}  &
          \text{if $c_l(j) = c_l(k)$} \\
          0 & \text{otherwise.}
        \end{cases}
      \end{align*}

      Write $g_{[l]}^{jk} = (g^{[l]})^{-1}_{jk}$.
      Applying the Sherman--Morrison formula to invert each block
      of $g^{[l]}$ gives
    \begin{equation}
\label{eqn:g_inv_recursion} 
      g_{[l]}^{jk} = \begin{cases}
        g_{[l - 1]}^{jk} - \frac{\alpha a_j a_k}{1 + \alpha b_{c_l(j)}} & \text{if $c_l(j) = c_l(k)$} \\
        0 & \text{otherwise,} \end{cases}
      \end{equation}
        where
        \[\alpha = \frac{\sigma_l^2}{1 + \sigma_l^2 s_{c_l(j)}} = \Theta\big((d^{c_l(j)}_l)^{-1}\big),\]
        \begin{equation}
          \label{eqn:a_definition}
          a_j = \sum_{k: c_l(k) = c_l(j)} r_j g_{[l-1]}^{jk}, \quad b_c = \sum_{j, k:  c_l(j) = c_l(k) = c} r_j r_k g_{[l-1]}^{jk}.
          \end{equation}

   We hypothesize that
    \begin{equation}
      \label{eqn:g_inv_sizes}
    g_{[l]}^{jk} = 
    \begin{cases}
  O_p(n_j^{-1}) & \text{if $j = k$} \\
  O_p\big((d^{3}_{c_3(j)})^{-1} n_j^{-1} n_k^{-1}\big) & \text{if $j \not = k$, but $c_3(j) = c_3(k)$} \\
 \vdots & \vdots \\
O_p\big((d^{l}_{c_l(j)})^{-1} n_j^{-1} n_k^{-1}\big) & \text{if $c_{l-1}(j) \not = c_{l-1}(k)$, but $c_{l}(j)  = c_{l}(k)$} \\
 0 & \text{otherwise,}
    \end{cases}
    \end{equation}
    and prove this by induction on $l$. This claim is true for $l = 2$, as
    $g^{[2]}$ is diagonal, with diagonal entries $h_{jj} + \sigma_2^{-2} = \Theta_p(n_j)$.
    For $l \geq 2$, applying the induction
    hypothesis to \eqref{eqn:a_definition}, recalling that $r_j = \Theta(1)$,
    we find
    \[a_j = \sum_{k: c_l(k) = c_l(j)} O_p\big((d^{l}_{c_l(j)})^{-1} n_j^{-1} n_k^{-1}\big)
    = O_p(n_j^{-1})\] and
    \[b_c =  \sum_{k: c_l(k) = c_l(j) = c} O_p\big((d^{l}_{c})^{-1} n_j^{-1} n_k^{-1}\big) = O_p(1),\]
    so
    \begin{equation}
      \label{eqn:g_inv_increment}
      \frac{\alpha \, a_j a_k}{1 + \alpha b_{c_l(j)}} =
      \frac{a_j a_k}{\alpha^{-1 }+ b_{c_l(j)}} = O_p\big((d^{c_l(j)}_l)^{-1} n_j^{-1} n_k^{-1}\big)
      \end{equation}
    Substituting \eqref{eqn:g_inv_increment} into \eqref{eqn:g_inv_recursion}
    proves \eqref{eqn:g_inv_sizes}.

    Normalizing,
    \begin{align*}
      f^{jk} &= n_j^{1/2} n_k^{1/2} g^{jk} = n_j^{1/2} n_k^{1/2} g^{jk}_{[L]} \\
      &=     \begin{cases}
  O_p(1) & \text{if $j = k$} \\
O_p\big((d^{l}_{c_l(j)})^{-1} n_j^{-1/2} n_k^{-1/2}\big) & \text{if $c_{l-1}(j) \not = c_{l-1}(k)$, but $c_{l}(j)  = c_{l}(k)$,} \\ &\text{ for $l = 3, \ldots, L$} \\
 0 & \text{otherwise.}
    \end{cases}
      \end{align*}
    Then
    \begin{align*}
      \sum_k \vert f^{jk} \vert
      &= O_p\left(1 + \sum_{l=3}^L \sum_{k : c_{l-1}(j) \not = c_{l-1}(k), c_l(j) = c_l(k)} \big(d^{l}_{c_l(j)}\big)^{-1} n_j^{-1/2} n_k^{-1/2} \right) \\
        & = O_p\left(1 + \sum_{l=3}^L d^{l}_{c_l(j)}  \big(d^{l}_{c_l(j)}\big)^{-1} n_j^{-1/2} \max_k \{n_k^{-1/2}\} \right) \\
        & = O_p(1 + n_j^{-1/2}) = O_p(1),
    \end{align*}
    and $\sum_j \vert f^{jk} \vert = \sum_{k} \vert f^{kj} \vert = O_p(1)$,
    so $f^{(2)} = O_p^{*}(1)$. So Condition
    \ref{cond:f_O*1} holds with normalizing terms $n_1, \ldots, n_d$,
    and Theorem \ref{thm:error_in_logintegral}
gives that $\epsilon_k({\theta}) = O_p(\sum_{j=1}^d n_j^{-k})$,
as required.
  \end{proof}
\end{appendices}

\bibliography{ogden}
\bibliographystyle{plainnat}

\end{document}